\documentclass[a4paper]{amsart}
\usepackage[T1]{fontenc}
\usepackage[utf8]{inputenc}

\usepackage{graphicx}
\usepackage{tikz-cd}
\usepackage{tikz}
\usetikzlibrary{positioning}
\usepackage{lipsum}
\usepackage{epigraph}
\usepackage{amsmath}
\usepackage{braket}
\usepackage{amssymb}
\usepackage{amsfonts}
\usepackage{amsthm}
\usepackage{booktabs}
\usepackage{bbm}
\usepackage{mathtools}

\usepackage[thinc]{esdiff}
\usepackage{pigpen}
\usepackage{hyperref}
\usepackage[capitalise]{cleveref}
\newcommand{\po}{\ar@{}[dr]|{\text{\pigpenfont R}}}
\newcommand{\pb}{\ar@{}[dr]|{\text{\pigpenfont J}}}

\theoremstyle{plain}
\newtheorem{thm}{Theorem}

\newtheorem{prop}[thm]{Proposition}
\newtheorem{cor}[thm]{Corollary}

\newtheorem{lem}[thm]{Lemma}

\newtheorem*{claim*}{Claim}

\theoremstyle{definition}
\newtheorem{defn}[thm]{Definition}
\newtheorem{rmk}[thm]{Remark}

\AddToHook{env/defn/begin}{\crefalias{thm}{defn}}
\AddToHook{env/lem/begin}{\crefalias{thm}{lem}}
\AddToHook{env/cor/begin}{\crefalias{thm}{cor}}
\AddToHook{env/prop/begin}{\crefalias{thm}{prop}}
\newcommand{\R}{\mathbb{R}}

\newcommand{\PP}{\mathcal{P}}

\newcommand{\Z}{\mathbb{Z}}
\newcommand{\SSS}{\mathbb{S}}

\makeatletter
\newsavebox{\@brx}
\newcommand{\llangle}[1][]{\savebox{\@brx}{\(\m@th{#1\langle}\)}
\mathopen{\copy\@brx\kern-0.5\wd\@brx\usebox{\@brx}}}
\newcommand{\rrangle}[1][]{\savebox{\@brx}{\(\m@th{#1\rangle}\)}
\mathclose{\copy\@brx\kern-0.5\wd\@brx\usebox{\@brx}}}
\makeatother
\title{Quasi-convex surface subgroups in some one-relator groups with torsion}
\author{Andrew Ng}
\date{}
\address{Mathematisches Institut, Universität Bonn, Endenicher Allee 60, 53115 Bonn, Germany}
\email{\href{mailto:clan@math.uni-bonn.de}{clan@math.uni-bonn.de}}
\begin{document}
\begin{abstract}
    We find surface subgroups in certain one-relator groups with torsion and use this to deduce a profinite criterion for a word in the free group to be primitive.

\end{abstract}
\maketitle
\section{Introduction}
A longstanding question often attributed to Gromov asks whether every one-ended hyperbolic group contains a subgroup isomorphic to the fundamental group of a closed hyperbolic surface, which from now on we will refer to as containing a surface subgroup. This has generated a lot of work on finding surface subgroups in various classes of hyperbolic groups. One of the most famous results in this vein is the existence of surface subgroups in fundamental groups of closed hyperbolic 3-manifolds \cite{KM12}. Recall that a hyperbolic group is said to be \emph{rigid} if it does not admit a nontrivial splitting with a virtually cyclic edge group. Another milestone in the theory is \cite[Corollary B]{W}, where Wilton proved that a one-ended hyperbolic group without $2$-torsion $\Gamma$ contains either a quasi-convex surface subgroup or a quasi-convex rigid subgroup, hence reducing the question of existence of surface subgroups for hyperbolic groups without 2-torsion to the question for one-ended rigid hyperbolic groups. We refer the reader to \cite{W} and the references therein for more background.

The aim of this note is to prove the following:
\begin{thm} \label{surfacesubgroup}
    Let $F_k$ be the non-abelian free group with free basis $\{x_1, \dots x_k\}$ and let $w$ be a word in the $x_i$ which is neither primitive nor a proper power. Then there exist positive integers $d$ and $N$ such that, for all $n\geq N$, the one-relator group $G_{dn}:=F_k/\llangle w^{dn}\rrangle$ contains a quasi-convex surface subgroup.
    
\end{thm}
Recall that a word $w$ in $F_k$ is said to be \emph{primitive} if $F_k$ splits as a free product $\langle w \rangle *F_{k-1}$. The assumption that $w$ isn't primitive can't be removed since when $w$ is primitive the corresponding one-relator groups are virtually free. 

There has also been substantial interest in characterising when a word is primitive. Whitehead famously gave an algorithm that can be used to decide whether a given word is primitive \cite[Ch.I.4]{LS77}. Later work of Puder and Parzanchevski showed that an element $w$ of $F_k$ which is primitive in $\widehat{F_k}$, the profinite completion of $F_k$, is already primitive in $F_k$ \cite{PP15}. This was later reproved and generalised by Wilton \cite[Corollary E]{W}. Using \cref{surfacesubgroup} we can give another profinite characterisation of when a word is primitive:
\begin{cor} \label{primitivity}
    Let $F_k$ be the non-abelian free group with free basis $\{x_1, \dots x_k\}$ and let $w$ be a word in the $x_i$ which is not a proper power. The following are equivalent:
    \begin{enumerate}
        \item $w$ is primitive;
        \item For all sufficiently large $n$ the one-relator group $G_n:=F_k/\llangle w^n \rrangle$ has the same profinite completion as the group $K_n:=F_{k-1}*\Z/n$.
    \end{enumerate}
\end{cor}
Recall that Remeslennikov's famous question \cite[Question 5.48]{KM14} asks whether all finitely generated residually finite groups that have the same profinite completion as a free group are in fact free. A positive answer to Remeslennikov's question would imply the stronger result that $w$ is primitive if and only if, for some $n$, $K_n$ and the one-relator group $G_n$ defined as in \cref{primitivity} have the same profinite completion. 

Indeed, suppose that for some $n$ the one-relator group $G_n:=F_k/\llangle w^n \rrangle$ has the same profinite completion as the group $K_n:=F_{k-1}*\Z/n$. Then $\widehat{G_n}$ is virtually a free profinite group, which would imply by Remeslennikov that $G_n$ is virtually free. 
We conclude that $w$ is primitive by \cite[Proposition II.5.13]{LS77}.

In the free group $F_{2k}$ we say that a word $w$ is a \emph{surface word} if $F/\llangle w \rrangle$ is isomorphic to the fundamental group of a closed non-positively curved surface. We can similarly recognise surface words:
\begin{cor} \label{surfaceword}
    A word $w$ which is not a proper power is a surface word in $F_{2k}$ if and only if, for all sufficiently large $n$, the one-relator group $G_n:=F_{2k}/\llangle w^n \rrangle$ is virtually a surface group.
\end{cor}

The strategy to prove \cref{surfacesubgroup} will be to apply a theorem of Wilton \cite[Theorem F]{W} that gives a suitable map of pairs from a surface with boundary to the free group such that the boundary is mapped to the word $w$ that we are interested in, modify the surface so that there are no accidental parabolics, and then use a result of Agol--Groves--Manning (\cref{noaccidents} below) to show that the fundamental group of the coned off orbifold injects into the one-relator group $G_n$ when $n$ is sufficiently large. We will recall below the relevant notions.

\section{Background}
\subsection{Relatively hyperbolic Dehn filling} \label{dehnfilling}
In this section we recall some facts from the theory of Dehn filling that we will need. We refer the reader to \cite{B} and \cite{AGM} and the references therein for background on relatively hyperbolic groups and Dehn filling.

    Let $G$ be a group which is hyperbolic relative to a finite collection of subgroups $\PP=\{P_1,\ldots,P_n\}$. Recall that a \emph{filling} of $G$ is a choice of subgroups $N_j \unlhd P_j$, called filling kernels. The quotient by the normal subgroup generated by the $N_j$ is denoted by $G(N_1,\dots N_m)$. 

\begin{defn}
    Let $H$ and $G$ be relatively hyperbolic groups relative to $\mathcal{Q}=\{Q_1, \ldots ,Q_m\}$ and $\PP=\{P_1,\ldots,P_n\}$ respectively. 
    We say that an element $g \in G$ (resp. $h \in H$) is a \emph{parabolic element} if it is contained in a parabolic subgroup of $G$ (resp. $H$).
    A homomorphism $\phi\colon H \to G$ is said to \emph{respect the peripheral structure (on $H$)} if, for every $i$, $\phi(Q_i)$ is conjugate in $G$ into some $P_j \in \PP$. Equivalently, the image of a parabolic element of $H$ is a parabolic element of $G$.
    An element $h \in H$ is said to be an \emph{accidental parabolic} element if $\phi(h)$ is parabolic in $G$ while $h$ is not parabolic in $H$.
    Building on this, we say that the map $\phi$ has \emph{no accidental parabolics} if it respects the peripheral structure and there are no accidental parabolic elements in $H$.
\end{defn}
\begin{defn}
    Suppose $G$ is a relatively hyperbolic group, relative to $\mathcal{P}$, and that $H < G$ is hyperbolic relative to $\mathcal{Q}$ and that the inclusion of $H$ into $G$ respects the peripheral structure. A filling $\varphi \colon G \to G(N_1,\dots N_m)$ is an \textit{$H$-filling} if, whenever $\varphi^{-1}(\varphi(H) \cap P_i^g$) is nontrivial, $N_i^ g \subseteq  sQ_j s^{-1} \subseteq H$ for some $s\in H$ and $Q_j \in \mathcal{Q}$.

\end{defn}

The main result we need from the theory of Dehn fillings is
\begin{prop} \cite[Proposition 4.4]{AGM} \label{noaccidents}
    Let $H<G$ be a relatively quasi-convex subgroup. For any sufficiently large $H$-filling $G(N_1,\dots N_m)$ of $G$, the induced map from the induced filling $H(K_1, \dots K_n)$ into $G(N_1,\dots N_m)$ is injective.
\end{prop}
\subsection{Orbifolds}
Orbifolds are a generalisation of manifolds where each point has a neighbourhood which is homeomorphic to an open set in the quotient of Euclidean space by the action of a finite group. A connected orbifold has a notion of dimension, which is the dimension $n$ of the Euclidean space as before. An orbifold of dimesion $n$ will be abbreviated as an $n$-orbifold. Orbifolds arise naturally as quotients of non-free group actions on manifolds. We refer the reader to \cite[Section 2]{sco} for a more detailed discussion and recall here the aspects of the theory that we need.

There is a notion of orbifold fundamental group and a bijection between orbifold coverings and subgroups of the orbifold fundamental group, akin to the Galois correspondence between topological covering spaces and subgroups of the fundamental group. It is therefore natural to ask when orbifolds have finite orbifold covers which are a manifold. 

\begin{defn}
    An orbifold is said to be bad if it has no finite cover which is a manifold. Orbifolds which aren't bad are said to be good. 
\end{defn}

In this article we will only be concerned with 2-orbifolds. Bad 2-orbifolds without boundary have been classified \cite[Theorem 2.3]{sco}. We state an immediate corollary of this classification that we will use.
\begin{prop} \label{goodorb}
    A 2-orbifold with at least three cone points or positive genus is good.
\end{prop}
We briefly recall what a cone point is. For $q>1$, the cyclic group $\Z/q\Z$ acts on $\R^2$ by rotation by angle $\frac{2\pi}{q}$. The image of the origin in the quotient space is the only point which is not a manifold point and is said to be a cone point.

More generally, an orbifold point $x$ is said to be a cone point if, for some $q>1$, $x$ has a neighbourhood isomorphic to the cone point of $\R^2/(\Z/q\Z)$.

\subsection{Relatively hyperbolic structures on hyperbolic groups} \label{subsec:relhypfree}

In order to apply the Dehn filling machinery we will view free groups as hyperbolic relative to certain subgroups.
    In \cite{B} it was shown that a hyperbolic group is hyperbolic relative to almost malnormal and quasi-convex subgroups. Since free groups are locally quasi-convex and maximal cyclic subgroups of a non-abelian free group are malnormal, the free group is hyperbolic relative to the subgroup $\langle w\rangle$, which from now on we take as the only parabolic of $F_k$. Since $F_k$ is locally quasi-convex, \cite[Theorem 1.1]{MP12} implies that any finitely generated subgroup of $F_k$ is relatively quasi-convex for this relatively hyperbolic structure.

Free groups also arise as the fundamental groups of compact surfaces with boundary. The inclusion of each boundary component induces an injection at the level of fundamental groups, and in general the fundamental group of a complete, finite volume hyperbolic manifold is relatively hyperbolic with parabolic subgroups the images of the boundary components \cite{Szc}.

These are related by the next theorem, which is a rephrasing of \cite[Theorem F]{W} that is adapted to the current setting. 

\begin{thm} \cite[Theorem F]{W} \label{wilton}
    Let $F_k$ be a free group and fix $w$ an imprimitive word which is not a proper power, so that $(F_k, \langle w \rangle)$ is a relatively hyperbolic structure on $F_k$. Represent this by a rose $R_k$ and denote by $S_w$ the loop representing $w$. Then there is a compact, connected, hyperbolic surface with boundary $\Sigma$ and a map of pairs of spaces $(\Sigma, \partial \Sigma) \to (R_k, S_w)$ inducing an injective map $\pi_1(\Sigma) \to F_k$ that respects the peripheral structure (when we view $\pi_1(\Sigma)$ as hyperbolic relative to its boundary components).
\end{thm}
    
\section{Proof of \cref{surfacesubgroup}} \label{surfaceproof}
First recall the following notion:
\begin{defn} \cite{AF}
    Let $F$ be a finitely generated free group, let $1\ne w \in F$ and let $H\le F$. An \emph{elevation} of $[w]_F$ to $H$ is a conjugacy class $[u]_H$ of an element $u\in H$ satisfying the following conditions:
    \begin{enumerate}
        \item $u=w_0^d$ for some $d\ge1$ and some representative $w_0\in [w]_F$.
        \item $d$ is the smallest integer for which $w_0^d\in H$.
    \end{enumerate}
    In that case, $d$ is called the \emph{degree} of the elevation $[u]_H$, denoted by $d=\deg_{[w]_F}[u]_H$ (or simply $d=\deg_{[w]}[u]$ when it is clear who $F$ and $H$ are).
\end{defn}

For the proof of \cref{surfacesubgroup} we will need a modified version of \cref{wilton}.
\begin{lem} \label{modified}
  Let $F_k$ be a free group and fix $w$ an imprimitive word which is not a proper power, so that $(F_k, \langle w \rangle)$ is a relatively hyperbolic structure on $F_k$. Then there is a compact, connected, hyperbolic surface with boundary $\Sigma'$ and an injective map $\phi: \pi_1(\Sigma') \to F_k$ that respects the peripheral structure (when we view $\pi_1(\Sigma)$ as hyperbolic relative to its boundary components) and has no accidental parabolics.
\end{lem}

To prove this, we will need the following intermediate result. 
In the setting of \cref{wilton}, let $\langle h \rangle \subset \pi_1(\Sigma)$ be a maximal cyclic subgroup such that $h$ is an accidental parabolic. We say that $h$ is a maximal accidental parabolic element.
\begin{lem} \label{pullback}
In the setting of \cref{wilton}, there are only finitely many maximal accidental parabolic elements.
\end{lem}
\begin{proof}

By \cite[Theorem 5.5]{S}, a maximal accidental parabolic element $g$ must belong to a conjugacy class given by an elevation of $w$ to $\pi_1(\Sigma)$. Denote by $S$ the result of applying Stallings' folding (see \cite[Section 3]{S}) to the 1-skeleton of $\Sigma$. Each maximal accidental parabolic is represented by a loop in the pullback graph $S \times_{R_k} S^1$ in the following pullback diagram:
\begin{center}
\begin{tikzcd}
S \times_{R_k} S^1 \arrow[r, "\hat{\iota}"] \arrow[d, "\hat{d}"] & S^1 \arrow[d, "d"] \\
S \arrow[r, "\iota"] & R_k
\end{tikzcd}
\end{center}
To conclude, note that setting $S_2$ to be the subgroup $\langle w \rangle \subset \pi_1(R_k)$ and $S_1$ to be $\pi_1(S)$ (which is finitely generated) in \cite[5.7(b)]{S}, we obtain that there are only finitely many elevations of $w$ to $\pi_1(S)$.
\end{proof}

\begin{proof}[Proof of \cref{modified}]
    
We will modify the surface given by \cref{wilton} to find a different hyperbolic surface without accidental parabolics. 

Suppose that $\langle g\rangle$ is a cyclic subgroup of $\pi_1(\Sigma)$ which isn't conjugate into the boundary but whose image under $f$ is conjugate in $F$ to a parabolic, i.e. a subgroup of a conjugate of $\langle w\rangle$. 

 There are only finitely many possible conjugacy classes of maximal accidental parabolic elements $g$ by \cref{pullback}. Choose representatives $g_1, g_2, \dots, g_n$ for these conjugacy classes. Note that if $h$ gives rise to an accidental parabolic then it will be a power of some elevation of $w$.

By results of \cite{Sco78}, there exists a finite cover $\hat{\Sigma}$ of $\Sigma$ where each conjugacy class is represented by an embedded curve. Up to passing to a further finite cover, we may take the cover to be normal, so all elevations of every conjugacy class are embedded curves in $\hat{\Sigma}$. Among all possible subsets of curves representing the elevations of the $g_i$ such that the images of any two curves are disjoint up to isotopy, choose some subset which is maximal under inclusion and cut along the curves in this subset. Since the original surface was hyperbolic, it couldn't have been an annulus, and in particular there is a connected component $C$ of the cut surface which is either a surface with positive genus or a sphere with at least three boundary components. We will take $\Sigma'$ to be $C$. 

It remains to explain why, when $C$ is given a relatively hyperbolic structure with its boundary components as parabolics, the map $f\colon \pi_1(C) \to F_k$ does not contain any accidental parabolics. This is because any accidental parabolic element $h \in \pi_1(C)$ which is not a proper power would have to be embedded in $\Sigma'$, and if it remained in $C$ as a curve which is not boundary parallel this would violate the maximality condition.
\end{proof}
We have now assembled all the ingredients we need.
\begin{proof} [Proof of \cref{surfacesubgroup}]
    Take the surface $\Sigma'$ from \cref{modified}. Let $d$ be the least common multiple of the degrees of the boundary components of $\Sigma'$. When $n$ is a multiple of $d$, the Dehn filling given by the filling kernel $\langle w^n \rangle$ is an $H$-filling. When $n$ is sufficiently large, \cref{noaccidents} implies that the fundamental group $H$ of the hyperbolic orbifold $O$ obtained by coning off the boundary components of $\Sigma'$ by a disc with cone point $\Z/n$ will inject in the corresponding Dehn filling of $F$, i.e. $G_n=F_k/\llangle w^n\rrangle$. 
    
   Since $\Sigma'$ has either $b\geq 3$ boundary components or genus $g>0$, $O$ is not on the list of bad orbifolds by \cref{goodorb}, hence it is finitely covered by a surface. Following section 2 of \cite{sco} we compute that the orbifold Euler characteristic is $\chi(O) = 2-2g-b(1-\frac{1}{n})$, which is negative when $n>3$. Since $O$ is good and Euler characteristic is multiplicative under finite covers, $O$ has a finite cover which is a closed surface of negative Euler characteristic, hence is hyperbolic. The fundamental group of this surface is the desired surface subgroup.

    Since it was proved in \cite{HW} that $G_n$ is locally quasi-convex when $n\geq |w|$, for sufficiently large $n$ the surface subgroup is quasi-convex. 
\end{proof}

\begin{rmk}
    Since all one-relator groups with torsion are hyperbolic by Newman's spelling theorem \cite{N68} it would be interesting to find surface subgroups in all of them, or at least with sufficiently large torsion. However, it seems difficult to control the degree of the maps on the boundary components of the surface from \cref{wilton}, and in general these will almost certainly have degree $>1$, so a different strategy will probably be necessary. By recent work of Kielak--Linton \cite{KL}, all one-relator groups with torsion are virtually free-by-cyclic, but it appears unclear how to leverage this fact; to the best of the author's knowledge it is unknown which free-by-cyclic groups contain surface subgroups. The most promising result in this direction seems to be the recent breakthrough of Wilton \cite{Wil26}, who showed that all fundamental groups of compact special cube complexes which are word hyperbolic but not free or surface groups contain one-ended subgroups of infinite index (c.f. the proof of \cref{primitivity} below).
\end{rmk}
\section{Proofs of the corollaries} \label{corproof}
\begin{proof} [Proof of \cref{primitivity}]
    One-relator groups with torsion are hyperbolic and cubulated when $n\geq 4$ \cite{LW13}, hence virtually compact special \cite{Agol} in the sense of \cite{HW08} and therefore cohomologically separable (also known as 'good in the sense of Serre') \cite[Proposition 3.2]{WZ}. If $w$ isn't primitive, for large enough $n$ there exists a quasi-convex surface subgroup $H$ in the corresponding one-relator group $G_n$ by \cref{surfacesubgroup}. Since quasi-convex subgroups of special groups are virtual retracts \cite[Corollary 7.9]{HW08}, there exists a finite index subgroup $G'$ of $G_n$ which retracts onto $H$. This implies that $G'$ has non-vanishing cohomology in degree 2 (with $\Z/2\Z$ coefficients), and since $G'$ is cohomologically separable its profinite completion $\widehat{G'}$ does too. This implies $\widehat{G} \ncong \widehat{K_n}$ since $\widehat{K_n}$ is virtually a free profinite group. 
\end{proof}

\begin{proof} [Proof of \cref{surfaceword}]
    Suppose first that $w$ is a surface word. Consider the orbifold $O$ obtained by taking the corresponding surface of genus $k$ with one boundary component and gluing in a disc with a cone point of order $n$. This orbifold is good by \cref{goodorb}, so is finitely covered by a surface. By applying the Seifert--van Kampen theorem for orbifolds \cite[Theorem 4.7.1]{Ch12}, we compute that $G_n$ is the orbifold fundamental group of $O$, so $G_n$ is virtually a surface group. 
    
    Now suppose $w$ isn't a surface word. If $w$ is primitive, then $G_n$ is virtually free, hence not virtually a surface. Otherwise, suppose $w$ is neither primitive nor a surface word. Then the surface subgroup given by \cref{wilton} is of infinite index. As in the proof of \cref{primitivity}, there is a finite index subgroup $G'$ of $G_n$ and a retraction $r\colon G' \to \pi_1(\Sigma')$ onto a surface subgroup of infinite index. The kernel is infinite, but one-relator groups are virtually torsion-free \cite{FKS72}, so the kernel contains an infinite cyclic subgroup. We now conclude verbatim as in the proof of \cite[Theorem 2]{Wil20}.
\end{proof}
\subsection*{Acknowledgements} The author is greatly indebted to Henry Wilton for suggesting this project and for countless helpful comments on earlier drafts which corrected inaccuracies and substantially improved the quality of exposition. Thanks are also due to Jonathan Fruchter for suggesting \cref{surfaceword}, his infectious enthusiasm and interest, and suggestions on how to improve the exposition. The author thanks the anonymous referee for carefully reading the paper and making a number of good suggestions. This work received funding from the European Union (ERC, SATURN, 101076148) and the Deutsche Forschungsgemeinschaft (EXC-2047/1 - 390685813).


\begin{thebibliography}{}
\bibitem[Agol13]{Agol}
I. Agol, The virtual Haken conjecture, Doc. Math. {\bf 18} (2013), 1045--1087; 
    \bibitem[AGM09]{AGM}
    I. Agol, D.~P. Groves and J.~F. Manning, Residual finiteness, QCERF and fillings of hyperbolic groups, Geom. Topol. {\bf 13} (2009), no.~2, 1043--1073; 
    \bibitem[AF]{AF}
    D. Ascari and J.~Fruchter, Virtual homological torsion in graphs of free groups with cyclic edge groups, arXiv:2505.20960
    \bibitem[Bow12]{B}
    B.~H. Bowditch, Relatively hyperbolic groups, Internat. J. Algebra Comput. {\bf 22} (2012), no.~3, 1250016, 66 pp.; 
    \bibitem[Ch12]{Ch12}
    S. Choi, {\it Geometric structures on 2-orbifolds: exploration of discrete symmetry}, MSJ Memoirs, 27, Math. Soc. Japan, Tokyo, 2012; 
    \bibitem[FKS72]{FKS72}
    J. Fischer, A. Karrass and D.~M. Solitar, On one-relator groups having elements of finite order, Proc. Amer. Math. Soc. {\bf 33} (1972), 297--301; 

    \bibitem[Ger83]{Ger83}
    S.~M. Gersten, Intersections of finitely generated subgroups of free groups and resolutions of graphs, Invent. Math. {\bf 71} (1983), no.~3, 567--591; 
    \bibitem[HW01]{HW}
    G.~C. Hruska and D.~T. Wise, Towers, ladders and the B. B. Newman spelling theorem, J. Aust. Math. Soc. {\bf 71} (2001), no.~1, 53--69; 
    \bibitem[HW08]{HW08}
    F. Haglund and D.~T. Wise, Special cube complexes, Geom. Funct. Anal. {\bf 17} (2008), no.~5, 1551--1620;
    \bibitem[KL]{KL}D. Kielak and M. Linton, Virtually free-by-cyclic groups, Geom. Funct. Anal. {\bf 34} (2024), no.~5, 1580--1608; 
    \bibitem[KM12]{KM12}
    J.~A. Kahn and V. Markovi\'c, Immersing almost geodesic surfaces in a closed hyperbolic three manifold, Ann. of Math. (2) {\bf 175} (2012), no.~3, 1127--1190; 
    \bibitem[KM14]{KM14}
     E. I., Khukhro and V.D. Mazurov (eds.), {\it The Kourovka notebook}, eighteenth edition, Russian Academy of Sciences Siberian Division, Institute of Mathematics, Novosibirsk, 2014; MR3408705
    \bibitem[LS77]{LS77}
    R.~C. Lyndon and P.~E. Schupp, {\it Combinatorial group theory}, reprint of the 1977 edition, 
Classics in Mathematics, Springer, Berlin, 2001; 
    \bibitem[LW13]{LW13}
    J. Lauer and D.~T. Wise, Cubulating one-relator groups with torsion, Math. Proc. Cambridge Philos. Soc. {\bf 155} (2013), no.~3, 411--429;
    \bibitem[MP12]{MP12}
    E. Mart\'inez-Pedroza, On quasiconvexity and relatively hyperbolic structures on groups, Geom. Dedicata {\bf 157} (2012), 269--290;
    \bibitem[MW05]{MW05}
    J.~P. McCammond and D.~T. Wise, Coherence, local quasiconvexity, and the perimeter of 2-complexes, Geom. Funct. Anal. {\bf 15} (2005), no.~4, 859--927; 
    \bibitem[N68]{N68}
    B.~B. Newman, Some results on one-relator groups, Bull. Amer. Math. Soc. {\bf 74} (1968), 568--571; 
    \bibitem[PP15]{PP15}
    D. Puder and O. Parzanchevski, Measure preserving words are primitive, J. Amer. Math. Soc. {\bf 28} (2015), no.~1, 63--97; 
    \bibitem[Sco78]{Sco78}
    G.~P. Scott, Subgroups of surface groups are almost geometric, J. London Math. Soc. (2) {\bf 17} (1978), no.~3, 555--565; 
    \bibitem[Sco83]{sco}
    G.~P. Scott, The geometries of $3$-manifolds, Bull. London Math. Soc. {\bf 15} (1983), no.~5, 401--487; 
    \bibitem[Sta83]{S}
    J.~R. Stallings, Topology of finite graphs, Invent. Math. {\bf 71} (1983), no.~3, 551--565; 
    \bibitem[Sz98]{Szc}
    A. Szczepa\'nski, Relatively hyperbolic groups, Michigan Math. J. {\bf 45} (1998), no.~3, 611--618; 
    \bibitem[Wil18]{W}
    H. Wilton, Essential surfaces in graph pairs, J. Amer. Math. Soc. {\bf 31} (2018), no.~4, 893--919; 
    \bibitem[Wil20]{Wil20}
    H. Wilton, On the profinite rigidity of surface groups and surface words, C. R. Math. Acad. Sci. Paris {\bf 359} (2021), 119--122; 
        \bibitem[Wil26]{Wil26}
    H. Wilton, Surface groups among cubulated hyperbolic and one-relator groups, Acta Mathematica, to appear; 
    \bibitem[WZ17]{WZ}
    H. Wilton and P.~A. Zalesskii, Distinguishing geometries using finite quotients, Geom. Topol. {\bf 21} (2017), no.~1, 345--384; 
\end{thebibliography}
\end{document}